\documentclass[11pt]{article}
\usepackage{jmlr2e_draft}

\usepackage{amsmath, amssymb} 
\usepackage{stmaryrd}
\usepackage{epsfig,calc, graphicx}
\usepackage{wasysym}
\usepackage[usenames]{color}
\usepackage{url}
\usepackage{enumitem}
\usepackage[lined,algoruled,algosection]{algorithm2e}
\usepackage{ifthen}
\usepackage{comment}
\newboolean{hideproof}
\setboolean{hideproof}{true}
\usepackage{mathtools}
\usepackage{bbm}
\usepackage{hyperref}
\usepackage{scalerel}
\usepackage{xcolor}
\usepackage{ifthen}

\newboolean{show_proof}
\usepackage{array}
\usepackage{booktabs}

\newboolean{arxiv}
\setboolean{arxiv}{true}

\newcommand\tp{t}
\newcommand\ip[2]{\langle #1, #2\rangle}

\newcommand{\wm}{\Lambda}
\newcommand{\hol}{H}

\newcommand\eps{\varepsilon}
\newcommand\dico{\Phi}

\newcommand\atom{\phi}
\newcommand\one{\mathbf{1}}

\newcommand\pdico{\Psi}

\newcommand\diag{\operatorname{diag}}

\newcommand{\N}{{\mathbb{N}}}
\newcommand{\R}{{\mathbb{R}}}

\newcommand{\E}{{\mathbb{E}}}

\renewcommand{\P}{{\mathbb{P}}}

\makeatletter
\providecommand*{\dcup}{%
  \mathbin{%
    \mathpalette\@dcup{}%
  }%
}
\newcommand*{\@dcup}[2]{%
  \ooalign{%
    $\m@th#1\cup$\cr
    \sbox0{$#1\cup$}%
    \dimen@=\ht0 %
    \sbox0{\scalebox{0.5}{$\m@th#1\circ$}}%
    \advance\dimen@ by -\ht0 %
    \dimen@=.5\dimen@
    \hidewidth\raise\dimen@\box0\hidewidth
  }%
}

\providecommand*{\bigdcup}{%
  \mathop{%
    \vphantom{\bigcup}%
    \mathpalette\@bigdcup{}%
  }%
}
\newcommand*{\@bigdcup}[2]{%
  \ooalign{%
    $\m@th#1\bigcup$\cr
    \sbox0{$#1\bigcup$}%
    \dimen@=\ht0 %
    \advance\dimen@ by -\dp0 %
    \sbox0{\scalebox{0.5}{$\m@th#1\circ$}}%
    \advance\dimen@ by -\ht0 %
    \dimen@=.5\dimen@
    \hidewidth\raise\dimen@\box0\hidewidth
  }%
}
\makeatother

\begin{document}

\title {Conditioning of incoherent sub-dictionaries sampled from a coherent dictionary
}

%\author{Karin Schnass
%\thanks{Karin Schnass is with the Department of Mathematics, University of Innsbruck, Technikerstra\ss e 13, 6020 Innsbruck, Austria,
%karin.schnass@uibk.ac.at.}
%}

%%jmlr 2 authors share affiliation
\author{\name Karin Schnass \email karin.schnass@uibk.ac.at\\
%%%co-authors in Innsbruck
    %\name Simon Ruetz \email simon.ruetz@uibk.ac.at\\	
	\addr  %Department of Mathematics\\
	Universit\"at Innsbruck\\
	Technikerstra\ss e 13\\
	6020 Innsbruck, Austria}

\editor{}

\maketitle
\begin{abstract}
Motivated by the desire to find a realistic and stable random model for $d$-dimensional signals, that are sparse in a transform-based and thus often coherent frame, such as a wavelet or a Gabor frame, we study the conditioning of incoherent sub-dictionaries sampled from a coherent dictionary, such as a unit norm frame.
In particular, we show that if the sub-dictionary is selected via a coherence rejective Poisson sampling model, it is well-conditioned with high probability, as long as its expected size  scales as $d/\log (K)$, where $K$ is the number of dictionary elements. The result is proved for the more general case of sampling quadratic sub-matrices from a real but not necessarily symmetric $K\times K$ matrix with zero diagonal, where coherence rejective sampling is defined via a symmetric mask, that acts as coherence substitute.  
\end{abstract}

% REQUIRED
%\begin{keywords}
%dictionary learning; sparse matrix factorisation; sparse coding; Method of Optimal Directions; MOD; Online Dictionary Learning; ODL; K-Singular Value Decomposition; K-SVD; convergence; non-uniform support distribution; rejective sampling
%\end{keywords}

% REQUIRED by SIAM
% \begin{AMS}
%  \textcolor{red}{ 68Q25, 68R10, 68U05}
% \end{AMS}

\section{Introduction}
 Redundant transform-based dictionaries such as wavelet and Gabor frames, are popular in approximation theory because they are known to provide compressible representations for many classes of signals, such as natural images or music, and because they can be handled efficiently using the Fourier transform, \cite{gabor46,da92,gr01,DONOHO:1998b}. On the other hand bases or dictionaries $\dico=(\atom_1, \ldots, \atom_K) \in \R^{d\times K}$ with $\|\atom_i\|_2=1$ and small coherence $\mu(\dico) = \max_{i\neq j} |\ip{\atom_i}{\atom_j}|$ are popular, because a signal $y$ that is sparse in $\dico$, meaning
\begin{align*}
     y= \sum_{i\in I} x_i \atom_i\quad \mbox{where} \quad S = |I|\ll d
\end{align*}
can again be easily approximated using for instance greedy algorithms and further because it can be efficiently handled, for instance via compressed sensing, \cite{parekr93,damazh94, do06cs,carota06}. \\
If redundant transform-based dictionaries are likely to provide compressible and thus almost sparse representations, and incoherent dictionaries allow for easy handling, the natural conclusion would be to work with redundant transform-based \emph{and} incoherent dictionaries. Unfortunately, this is not possible.
So while wavelet and Gabor frames have more approximation power the larger they are, they also become more coherent, 
$\mu(\dico) \approx 1$, and thus hard to handle in a sparse context. Indeed, choosing an ill-conditioned sub-dictionary $\dico_I =(\atom_{i_1},\ldots \atom_{i_S})$ for $I=\{i_1,\ldots, i_S\}$, which results in an unstable signal, where $\|y\|_2= \|\dico_I x_I\|_2 \ll \|x_I\|_2$ or $\|x_I\|_2 \ll \|y\|_2$, becomes very easy if the dictionary is coherent, so results valid for all $S$-sparse signals become impossible for interesting ranges of $S$, such as $S>2$. Worse, even if we select the sub-dictionary at random either uniformly or according to some natural distribution which prefers lower frequency atoms to higher frequency ones, we will with high probability select a coherent and ill-conditioned sub-dictionary, so also results valid for most $S$-sparse signals become impossible.\\
In this paper we try to work around the problems described above and study the selection of incoherent sub-dictionaries from a coherent dictionary, according to a coherence-rejective Poisson sampling model. In particular, we will derive results of the following form.

\begin{ftheorem}\label{featured_theorem}
Let $\dico$ be a $d\times K$ dictionary and $p=S/K \in (0,1)$.  Assume that we select random sub-dictionaries $\dico_I$ according to a distribution $\P_\mu$, where $\P_\mu(I = J)= c \cdot p^{|J|} \cdot (1-p)^{K-|J|}$ if $\mu(\dico_J) = \max_{i,j\in J,i\neq j}|\ip{\atom_i}{\atom_j}|\leq \mu_0$ and $\P_\mu(I = J)= 0$ else, then we have for all $r\geq 6 \|\dico\|^2\cdot S/K$
\begin{align*}
\P_\mu(\|\dico_I^\tp \dico_I - \mathbb{I}\| > r )  \leq 54 K \exp\left( - \min\left\{ \frac{r}{6 \mu_0}, \frac{K r^2  }{36 S \|\dico\|^2 } \right\} \right).
\end{align*}
\end{ftheorem}
In the next section we clarify the notation used throughout the paper, introduce the coherence rejective sampling model and collect several auxiliary results, which are proved in the appendix. Finally, Section~\ref{sec_main_result} contains the main result in full generality as well as its proof. %We conclude with a short discussion, where we point out applications and future research directions.

\section{Notation and Sampling Model}
For a (sorted) set $I = \{i_1,\ldots,i_S\}\subseteq [K]$, meaning $i_1<i_2 <\ldots < i_S$, and any $d\times K$ matrix~$\dico$ we define the sub-matrix $\dico_I$ as the restriction of $\dico$ to the columns indexed by $I$, that is $\dico_I = (\atom_{i_1}, \ldots ,\atom_{i_S})$. We also define the restriction operator $R_I$ as the restriction of the $K\times K$ identity matrix to $I$, meaning $R_I = \mathbb{I}_I\in \R^{K\times S}$ and leading to $\dico_I = \dico R_I$.
When applying restriction and transpose to a matrix we will use the convention that restriction acts first, meaning $\dico_I^\tp = (\dico_I)^\tp$ and $R_I^\tp = (R_I)^\tp$.
Note that $R_I^\tp$ applied from the left restricts to the rows indexed by $I$, while $R_I$ and $R_I^\tp$ applied from the left  and right respectively act as zero-padding operators. So, for example, $D_I$ the diagonal matrix, whose $i$-th diagonal entry is one if $i\in I$ and zero else, can be written as $D_I = R_I R_I^\tp$.\\
We define the $p,q$-norms for a real matrix $A$ via, 
\begin{align*}
    \|A\|_{p,q} &=\max_{v : \|v\|_q=1} \|Av\|_p.
\end{align*} 
The operator norm $\|\cdot\|_{2,2}$ coincides with the largest eigenvalue in absolute and will also be denoted by $\|\cdot\|$. Note that for a square matrix $H$, $H_{I,I}:=R_I^\tp H R_I$ and $D_I H D_I$ differ in size but have the same operator norm. Further, we recall that $\|A\|_{2,1} = \max_k \|A e_k\|_2$ is the largest Euklidean norm of a column, $\|A\|_{\infty,2} = \max_k \|e_k^\tp A\|_2=\max_k \|A^\tp e_k\|_2$ is the largest Euklidean norm of a row, and $\|A\|_{\infty,1} = \max_{ij}|A_{ij}|$ is the largest entry in absolute. \\
Finally, in order to extend the concept of incoherence to not necessarily symmetric matrices $H$ with zero diagonal, we will make use of masks $M \in \{0,1\}^{K\times K}$, where $M_{ij} = 1$ indicates that the entries $i,j$ are coherent to each other. So in case $H$ is the hollow Gram matrix of a dictionary, $H=\dico^\tp\dico - \mathbb{I}$, once we select a cut-off threshold $\mu_0$ we can set $M_{ij} = 0$ if $i=j$ or $|\ip{\atom_i}{\atom_j}|\leq \mu_0$ and $1$ else.
We can now define coherence rejective sampling in full generality, which in particular allows for some atoms $\atom_i$, corresponding to rows/columns in $H$ to be selected with higher probability, as would be natural for low frequency wavelets or Gabor atoms.

\begin{definition}[Coherence rejective sampling]\label{coh_rej_sampling}
\phantom{bla}\\
Assume we are given $K\in \N$, a sequence $p\in (0,1)^K$ and a symmetric matrix $M\in \{0,1\}^{K\times K}$ with $M_{ij}\in\{0,1\}$ and $M_{ii}=0$ for all $i,j\in[K]$, then we define the following discrete probability measures and random variables on $\Delta = \{0,1\}^K$.
For $A\subseteq \Delta$ we set
\begin{align}
\P(A) = \sum_{\delta \in A} p^\delta (\one -p)^{\one - \delta} = \sum_{\delta \in A} \: \prod_{i:\delta_i =1} p_i \prod_{i:\delta_i =0} (1-p_i).
\end{align}
We next define the set-valued random variable $I$ and the real-valued random variable $\mu$, as
\begin{align}
\begin{array}{llcl}
I:& \Delta &\to&  \mathcal{P}([K])\\
& \delta &\mapsto &\{i : \delta_i = 1\}
\end{array}\quad 
\begin{array}{llcl}
\mu: & \Delta &\to& \R \\
& \delta &\mapsto &\delta^t M \delta
\end{array}
\end{align}
and based on $\mu,I$ and a set $S\subseteq [K]$ the following discrete probability measures on $\Delta$, 
\begin{align}
    \P^S(\cdot) = \P(\cdot \mid I\subseteq S), \quad 
     \P_\mu(\cdot) = \P(\cdot \mid \mu = 0) \quad \mbox{and} \quad
     \P^S_\mu(\cdot) = \P(\cdot \mid \mu = 0, I \subseteq S).
\end{align}
The probability measure $\P$ is called Bernoulli or Poisson sampling. Further, we refer to $\P^S$ as $S$-trace (Bernoulli/Poisson) sampling and to $\P^S_\mu$ as coherence rejective ($S$-trace) sampling.
\end{definition}
Note that the random variable $I$ is bijective, so any real-valued random variable $g$ on $\Delta$ has an analogue random variable $\tilde g: \mathcal{P}([K])\to \R, J\mapsto g(\one_J)$, such that $g = \tilde g \circ I.$
In the case of $\mu$ we will by abuse of notation identify $\tilde\mu$ with $\mu$. We will also extend $\mu$ to take two inputs. So for $\delta,\eps \in \Delta$ and $J,F \in \mathcal{P}[K]$ we set $\mu(\delta, \eps) =  \delta^t M \eps $ and $\mu(J,F)= \one_J^t M \one_F$ respectively. We will call a set incoherent if $\mu(J)=0$, and a set $J$ and index $k$ incoherent to a set $F$ if $\mu(F,J)=0$ and $\mu(F,k)=\mu(F,\{k\})= 0$ respectively. Whenever $\mu(\cdot,\cdot) > 0$ we refer to the set or pair as coherent.\\
\noindent
Before we state and prove our theorem we need to collect four auxiliary results.
The first tells us that if we condition coherence rejective sampling on the outcome of intersecting $I$ with a pre-chosen set $S$, i.e. on the event $I\cup S = F$, then we essentially get coherence rejective trace sampling on a set determined by $F$.

\begin{lemma} \label{cond_on_F}
Let $g$ be a real-valued random variable on $\mathcal P([K])$, $S\subseteq [K]$ and $F\subseteq S$ with $\mu(F) = 0$. 
Let $A_F$ consist of all indices outside $S$ that are incoherent to the selected set $F$, that is, $A_F := \{ k \in S^c \mid \mu(F,k)=0\}$, 
then for any $r\in \R$
\begin{align*}
    \P_\mu(g(I) \geq r \mid I\cap S = F) = \P_\mu^{A_F}(g(I\cup F)\geq r).
\end{align*}
\end{lemma}
The second result allows us to bound the inclusion probabilities of $\P_\mu^S$ by those of $\P$.
\begin{lemma}\label{bound_inclusion_prob}
Let $\chi_\ell: \Delta\to \R, \delta \mapsto \delta_\ell$ then
for any $S\subseteq [K]$ we have
\begin{align*}
    \E^S_\mu(\chi_\ell) = \P^S_\mu(\chi_\ell = 1) = \P^S_\mu(\ell \in I) \leq p_\ell.
\end{align*}
\end{lemma}
The proof of both lemmas can be found in Appendix~\ref{sec:proof_lemma}.\\
Next, we need the following proposition, which is essentially a summary and slight modification of results by Tropp and Chretien \& Darses, \cite{tr08, chda12}, who in turn refer to Bourgain \& Tzafriri and de la Peña \& Giné, \cite{botz87, DecouplingBook}. 

\begin{proposition}[Tropp, Chretien \& Darses]\label{prop_decoupling}
Let $\hat \P$ be a (discrete) probability measure on $\Delta$, $I$ the set-valued random variable with $I(\delta) = \{i:\delta_i=1\}$ and $\hol \in \R^{K\times K}$ a matrix with zero diagonal, then there exists a subset $T\subseteq [K]$, such that for all $\theta \in \R$
\begin{align*}
    \hat\P( \| D_I \hol D_I \| \geq r )&\leq  36 \cdot \hat\P\left( \max \left\{\| D_I \, D_T \hol D_{T^c} \, D_I \|, \| D_I \, D_{T^c} \hol D_{T} \, D_I \| \right\}\geq \frac{r}{2}\right) . %\notag\\
    %&=36 \cdot \hat \P \left( \max \left\{ \| D_{I \cap T}\, H \,D_{I\cap T^c} \|, \| D_{I \cap T^c}\, H \,D_{I\cap T} \|\right\} \geq \frac{r}{2}\right ).
\end{align*}
\end{proposition}
\begin{proof} Follow the proof of \cite[Proposition~4.1]{chda12} and note that the operator norm satisfies
$$\left\|\begin{pmatrix} 0 &  A \\ B & 0\end{pmatrix}\right\| =\max\{\|A\|, \|B\|\}.$$ Alternatively, follow the proof of \cite[Proposition~A.1]{rusc21} and skip the step of obtaining an independent copy $\delta'$ of $\delta$.
\end{proof}
Finally, we need the following version of Freedman's inequality for matrices, which is a simplification of \cite[Theorem~1.2]{tropp_freedman}.

\begin{corollary}[Matrix Freedman]
Let $(Y_\ell)_{\ell}$ with $Y_0=0$ be a matrix martingale adapted to the filtration $(\mathcal F_\ell)_\ell$, whose values are symmetric (self-adjoint) $d\times d$ matrices.
Assume that the difference sequence $X_\ell = Y_\ell-Y_{\ell-1}$ is almost surely bounded by $R$ and that the quadratic variation process is almost surely bounded by $\sigma^2$, meaning 
\begin{align*}
    \|X_\ell\| \leq R \quad  \mbox{and} \quad & \| \sum_{\ell=1}^L \E[X_\ell^2\mid \mathcal F_{\ell-1}]\| \leq \sigma^2 \quad \mbox{(a.s.), then}\\
\P(\| Y_L \|\geq t) = \P\Big(\| \sum_{\ell=1}^L X_\ell \| &\geq t \Big)\leq d \cdot \exp\left(\frac{-t^2/2}{\sigma^2+ Rt/3}\right)\leq d \cdot \exp\left(\frac{-t^2}{2\sigma^2+ Rt}\right).
\end{align*}
\end{corollary}
We are now ready to state and prove our main theorem.

\section{Main result}\label{sec_main_result}
The following result is the generalisation of Featured Theorem~\ref{featured_theorem} to non-symmetric matrices with zero diagonals, e.g. hollow cross-Gram matrices $\pdico^\tp \dico - \diag(\pdico^\tp \dico)$, where $\pdico$ is a perturbed version of $\dico$. Accordingly, it allows for a more general definition of (in)coherence. Last but not least it considers selection probabilities that can vary between atoms. 
\begin{theorem} Let $H$ be a real $K\times K$ matrix with zero diagonal, $M$ a symmetric $K\times K$ matrix with zero diagonal and entries in $\{0,1\}$ and $p\in (0,1)^K$. 
Further, let $\wm$ be the diagonal matrix with $\wm_{ii} = \sqrt{p_i}$, $\bar M$ be the complement matrix to $M$, ie. $\bar M_{ij}: = 1- M_{ij}$, and $\hat H$ be any matrix that coincides with $H$ on $\bar M$, that is for the Hadamard (entry-wise) product we have $\bar M \odot \hat H = \bar M \odot H$. Denoting with $\P_\mu$ the coherence rejective sampling measure associated with $M$ and $p$ from Definition~\ref{coh_rej_sampling}, we then have for any $r\geq 6 \|\wm \hat H \wm \|$ 
\begin{align*}
\P_\mu(\|D_I H D_I\| \geq r) \leq 108 K \exp\left( - \min\left\{ \frac{r}{6 \|\bar M \odot\hat H\|_{\infty,1}}, \frac{r^2}{36\|\wm \hat H\|_{2,1}^2 },\frac{r^2}{36\| \hat H \wm \|_{\infty,2}^2 }\right\} \right).
\end{align*}
\end{theorem}
Before we go to the proof, note the last two norms in the bound above are simply the largest column and row norm respectively and thus coincide when $\hat H$ is symmetric, ie. $\|\wm \hat H\|_{2,1} = \max_\ell \|\wm \hat H e_\ell\|_2$ and $\|\hat H \wm \|_{\infty,2}=\max_\ell \| e_\ell^\tp \hat H \wm \|_2$. Further note that $\|\bar M \odot\hat H\|_{\infty,1} = \max_{(i,j):M_{ij}=0}|H_{ij}|$. Thus if $H$ is symmetric and we define $M$ via a coherence cut-off threshold $\mu_0$, ie. $M_{ij}=1$ if $|H_{ij}|>\mu_0$ and zero else, we simply get $\|\bar M \odot\hat H\|_{\infty,1} \leq \mu_0$. In the symmetric case, clearly $\|\wm \hat H\|_{2,1}$ is minimised for the natural choice $\hat H =  \bar M \odot H: =\bar H$. However, this might come at the cost of an increased operator norm $\|\wm \bar H \wm\|$ and one might opt for the other natural choice $\hat H = H$, as in Featured Theorem \ref{featured_theorem}. To get to the statement there, we first use that for $H$ and $\hat H$ symmetric, the constant 108 reduces to 54, and so choosing $\hat H = H = \dico^\tp \dico -\mathbb{I}$ we have for $p\equiv S/K$, 
\begin{align*}
\|\wm \hat H\|^2_{2,1} &= S/K\cdot \max_k \|\dico^\tp \atom_k - e_k \|^2_2 \leq S/K \cdot \|\dico^\tp\|^2 \cdot \|\atom_k\|^2 = S/K \cdot \|\dico\|^2 ,\\
\|\wm \hat H \wm \| & = S/K \cdot \| \dico^\tp \dico -\mathbb{I}\| \leq S/K \cdot \| \dico^\tp \dico\| = S/K \cdot \| \dico\|^2,
\end{align*}
where for the second bound we have used that $\dico^\tp \dico$ is positive semi-definite with largest eigenvalue $\geq 1$.\\
The formulation with $\hat H$ is expected to be most valuable in the case of transform based dictionaries, where the structure can be exploited to find $\hat H$ with maximal column norm similar to $\bar H$ and operator norm much smaller than $H$.

\begin{proof}
We first prove the theorem for the inarguably most interesting case where $H$ is symmetric. Note that for $H$ symmetric and any $S\subseteq [K]$ we have $\| D_I \, D_S H D_{S^c} \, D_I\| = \| D_I \, D_{S^c} H D_{S} \, D_I\|$ so by Proposition~\ref{prop_decoupling} we know that there exists $T\subseteq [K]$ with $|T|\geq |T^c|$ such that
\begin{align*}
    \P_\mu( \| D_I H D_I \| \geq r )&\leq  36 \cdot \P_\mu\left(\| D_I \, D_T H D_{T^c} \, D_I \|\geq \frac{r}{2}\right).  %\notag\\
    %&=36 \cdot \hat \P \left( \max \left\{ \| D_{I \cap T}\, H \,D_{I\cap T^c} \|, \| D_{I \cap T^c}\, H \,D_{I\cap T} \|\right\} \geq \frac{r}{2}\right ).
\end{align*}
Conditioning on the intersection of $I$ with $T$ and applying Lemma~\ref{cond_on_F} with $S=T$ and $g(I) = \| D_I \, D_T H D_{T^c} \, D_I \| $, yields
\begin{align}
\P_{\mu} \left( \| D_I \, D_T H D_{T^c} \, D_I\| \geq r \right ) &= \sum_{\substack{F:F\subseteq T\\ \mu(F)=0}} \P_{\mu} \left( \| D_F\, H D_{T^c}\,D_I \| \geq r \mid I\cap T =F \right ) \cdot \P_\mu(I \cap T = F)\notag \\
&= \sum_{\substack{F:F\subseteq T\\\mu(F)=0}} \P^{A_F}_{\mu} \left( \| D_F\, H \,D_{T^c} D_{I\cup F}\| \geq r \right ) \cdot \P_\mu(I \cap T = F),\notag \\
&= \sum_{\substack{F:F\subseteq T}} \P^{A_F}_{\mu} \left( \| D_F\, H D_{A_F} \,D_I\| \geq r \right ) \cdot \P_\mu(I \cap T = F),\label{total_prob_cond_on_F}
%&= \sum_{\substack{F:F\subseteq T}} \P^{A_F}_{\mu} \left( \| D_F\, H \,D_J \| \geq r \right ) \cdot \P_\mu(I \cap T = F).
\end{align}
where we have used that $\P_\mu(I \cap T = F) = 0 $ if $\mu(F)>0$. Indeed, since $\mu(F)\leq \mu(J)$ whenever $F\subseteq J$, if $\mu(F) > 0$, there exists no $\delta$ with $I(\delta) \cap T = F$ and $\mu(\delta) =\mu(I(\delta)) = 0 $.\\
Let $\bar H = H\odot \bar M$, meaning we set all entries of $H$ corresponding to coherent index pairs to zero. This means that $\bar H, H$ and $\hat H$ coincide for all incoherent index pairs, so in particular,
$$D_F^\tp\, \bar H D_{A_F}=D_F^\tp\, H D_{A_F} = D_F^\tp \, \hat H 
D_{A_F}.$$ 
Further, recall that $D_F$ acts like $R_F^t$ plus zero-padding, so for all matrices $B$ we have $\|D_F B\|_{p,q} = \|R_F^tB\|_{p,q}$. In particular it suffices to bound
$$ \P^{A_F}_{\mu}( \| R^t_F\, H D_{A_F} \,D_I\| \geq r).$$
In order to do this with minimal chaos, for a fixed $F$ we rearrange $p$, $H$, $\bar H$, $\hat H$ and $M$ such that $A_F = \{1,\ldots ,L\}=[L]$ and abbreviate $R_F^t\, H D_{A_F} = B = (B_1, \ldots, B_L, 0,\dots ,0)\in \R^{|F|\times K}$, meaning $\P^{A_F}_\mu$ can be identified with $\P_\mu$ on $\mathcal P([L])$ based on $p|_{[L]}$ and $M_{[L],[L]}$ and our task reduces to bounding
\begin{align*}
\P_{\mu} \left( \| B \,D_I \| \geq r \right) &=\P_{\mu} \left( \|B \,D_I\, B^\tp\| \geq r^2 \right )\\
&=\P_{\mu} \left( \|\sum_{\ell \in I} B_\ell  B_\ell ^\tp\| \geq r^2 \right )=\P_{\mu} \left( \|\sum_{\ell = 1}^L \chi_\ell B_\ell B_\ell^\tp\| \geq r^2 \right ),
\end{align*}
where we recall that $\chi_\ell(\delta) =\delta_\ell$.
Next we rewrite the sum we want to bound as a martingale adapted to the filtration $\mathcal{F}_\ell = \sigma(\chi_1, \ldots, \chi_\ell)$. For any $\ell \in [L]$ we define the following random variables on $\Delta$,
\begin{align}
\pi_{\ell} : = \E_\mu[\chi_\ell \mid \mathcal{F}_{\ell-1}]\quad \mbox{and} \quad 
X_\ell:= (\chi_\ell- \pi_{\ell})\cdot B_\ell B_\ell^\tp.
%\mbox{and}\quad F_\ell&: = \{j\leq \ell : \chi_j=1\} =I \cap [\ell]
\end{align}
Recall that $\tilde \chi_\ell$ acts on $\mathcal P([L])$ by $\tilde \chi_\ell (J) = \chi_\ell(\one_J)$, meaning $\chi_\ell = \tilde \chi_\ell \circ I$. Using Lemma~\ref{cond_on_F} with $g=\tilde \chi_\ell$, $S=[\ell - 1]=\{1,\ldots, \ell-1\}$ and $F=I(\delta) \cap [\ell-1]$ in combination with Lemma~\ref{bound_inclusion_prob} then yields for any $\delta$ with $\mu(\delta)=0$,
%We also set $T_\ell := T\cup [\ell]$,  and $F_\ell = F \cup \{i\leq \ell : \delta_\ell=1\} $ and $A_\ell$ the largest subset of $T_\ell^c$ such that $\mu(F_\ell,A_\ell)\leq \mu_0$, that is 
%\begin{align} 
%A_\ell &:=  \{i\in T_\ell^c \mid \max_{j\in F_\ell} \absip{\atom_i}{\atom_j} \leq \mu_0\}.
%\end{align}
%Combining \eqref{PmuF_reform1} and \eqref{cond_inc_prob} we can then bound $\pi_\ell$ almost surely as
\begin{align*}
\pi_{\ell}(\delta) 
&= \E_\mu\,[ \chi_\ell \mid \chi_1 = \delta_1 \ldots \chi_{\ell-1} = \delta_{\ell-1}] \\
&= \E_\mu\,[ \tilde \chi_\ell \circ I \mid I\cap S = F ]\\
&= \P_\mu\,( \tilde \chi_\ell ( I) \geq 1/2 \mid I\cap S = F) \\
&= \P^{A_F}_\mu(\tilde \chi_\ell (I\cup F) \geq 1/2) = \P^{A_F}_\mu(\chi_\ell \geq 1/2) =\E^{A_F}_\mu(\chi_\ell )\leq p_\ell,
\end{align*}
or in other words $\pi_\ell \leq p_\ell$ almost surely.
This further leads to 
\begin{align*}
\sum_{\ell = 1}^L \chi_\ell B_\ell B_\ell^\tp &= \sum_{\ell = 1}^L (\chi_\ell- \pi_\ell)\cdot B_\ell B_\ell^\tp  +  \sum_{\ell = 1}^L \pi_\ell \cdot B_\ell B_\ell^\tp \preceq \sum_{\ell = 1}^L X_\ell+  \sum_{\ell = 1}^L p_\ell \cdot B_\ell B_\ell^\tp \quad \mbox{(a.s.).} %= \sum_{\ell = 1}^L X_\ell+ B \wm D_{[L]} \wm B^\tp,
\end{align*}
Let $\wm$ be the $K\times K$ diagonal matrix with $\wm_{ii}=\sqrt{p_i}$, then we have $\sum_{\ell = 1}^L p_\ell \cdot B_\ell B_\ell^\tp = B\wm D_{[L]} \wm B^\tp$, which we can bound as
\begin{align*}
\| B \wm D_{[L]} \wm B^\tp\| =\| B \wm D_{[L]} \|^2&= \| D_F H \wm D_{A_F}\|^2 \\
&= \| D_F \hat H \wm D_{A_F}\|^2\leq \| D_F \hat H \wm D_{T^c}\|^2  = :\hat c_F^2.
\end{align*}
This allows us to further estimate
\begin{align*}
\P_{\mu} \left( \|BD_I \| \geq r \right) &=\P_{\mu} \left( \|\sum_{\ell = 1}^L \chi_\ell B_\ell B_\ell^\tp\| \geq r^2 \right )\leq \P_{\mu} \left( \|\sum_{\ell = 1}^LX_\ell\|  \geq r^2 - \hat c_F^2 \right ).
\end{align*}
By construction $(Y_j)_j$ with $Y_j =  \sum_{\ell = 1}^j X_\ell$ is a martingale with respect to the filtration $\mathcal{F}_j$, so we can use Freedman's inequality to estimate $\P_\mu(\| \sum_{\ell = 1}^L X_\ell\| \geq t)$.
Concretely, we can bound the difference sequence $X_\ell$ almost surely as
\begin{align}
\| X_\ell\| = |\chi_\ell- \pi_\ell| \cdot  \|B_\ell\|_2^2 \leq \max_{\ell \in A_F} \|B_\ell\|_2^2 &= \|D_F H D_{A_F}\|^2_{2,1}\notag \\
&= \|D_F \bar H D_{A_F}\|^2_{2,1} \leq \|D_F \bar H D_{T^c}\|^2_{2,1}=:\bar a^2_F .
\end{align}
Further, we get for the predictable quadratic variation process
\begin{align*}
\sum_{\ell=1}^L \E_\mu[ X_\ell^2 \mid \mathcal{F}_{\ell-1}]& =\sum_{\ell=1}^L  \E_\mu[(\chi_\ell- \pi_\ell)^2 \cdot B_\ell B_\ell^\tp B_\ell B_\ell^\tp \mid \mathcal{F}_{\ell-1}] \\
&= \sum_{\ell=1}^L \E_\mu[\chi_\ell- 2\chi_\ell \cdot \pi_\ell + \pi_\ell^2 \mid \mathcal{F}_{\ell-1}] \cdot \|B_\ell\|_2^2 \cdot B_\ell B_\ell^\tp \\
&= \sum_{\ell=1}^L  (\pi_\ell - \pi_\ell^2) \cdot \|B_\ell\|_2^2 \cdot B_\ell B_\ell^\tp \\
%&\preceq  \sum_{\ell=1}^L  p_\ell \cdot \|B_\ell\|_2^2 \cdot B_\ell B_\ell^\tp  \\
&\preceq  \max_{\ell \in [L]}  \|B_\ell\|_2^2  \cdot \sum_{\ell=1}^L  p_\ell \cdot B_\ell B_\ell^\tp = \|B D_{A_F}\|^2_{2,1}
\cdot B \wm D_{A_F} \wm B^\tp \quad \mbox{(a.s.)},
\end{align*}
and therefore
\begin{align}
\| \sum_{\ell=1}^L \E_\mu[X_\ell^2 \mid \mathcal{F}_{\ell-1}]\| &\leq \|B D_{A_F}\|^2_{2,1}
\cdot \| B \wm D_{A_F}\|^2 \notag \\
&= \|D_F \bar H D_{A_F}\|^2_{2,1}  \cdot \| D_F \hat H \wm D_{A_F} \|^2 \leq \bar a^2_F \cdot \hat c_F^2 \quad \mbox{(a.s.)}.\label{quad_var_bound}
\end{align}
Thus using Freedman's inequality and noting that $|F|\leq |T|$, we get
\begin{align}
  \P_{\mu} \left( \| B \,D_{I} \| \geq r \right) &\leq \P_\mu ( \| \sum_{\ell = 1}^L X_\ell\| \geq r^2 - \hat c_F^2 ) \notag \\
  &\leq  | F | \cdot \exp\left( \frac{-(r^2 - \hat c_F^2)^2}{2 \cdot \bar a^2_F \cdot \hat c_F^2 + \frac{2}{3} \cdot \bar a^2_F\cdot (r^2 - \hat c_F^2) }\right) 
\leq  | T | \cdot \exp\left( \frac{-(r^2 - \hat c_F^2)^2}{ \bar a^2_F \cdot ( r^2 + \hat c_F^2) }\right).\notag %\label{firstexponential}
\end{align}
Substituting the bound above back into \eqref{total_prob_cond_on_F} after partitioning the sets $F$ conveniently, yields
\begin{align}
\P_{\mu} \left( \| D_I \, D_T H D_{T^c} \, D_I\| \geq r \right )& = \sum_{\substack{F:F\subseteq T}} \P^{A_F}_{\mu} \left( \| R_F^t\, H D_{A_F} \,D_I\| \geq r \right ) \cdot \P_\mu(I \cap T = F)\notag \\
&\hspace{-3cm}\leq \sum_{\substack{F:F\subseteq T\\ \hat c_F \leq s, \bar a_F \leq t}}  | T | \cdot \exp\left( \frac{- (r^2 - s^2)^2}{t^2 \cdot ( r^2 + s^2) }\right)\cdot  \P_\mu(I \cap T = F) \notag \\
&\qquad \qquad 
+ \sum_{\substack{F:F\subseteq T\\\hat c_F >s}} 1 \cdot \P_\mu(I \cap T = F) + \sum_{\substack{F:F\subseteq T\\\bar a_F > t}} 1 \cdot \P_\mu(I \cap T = F)\notag \\
&\hspace{-3cm}\leq  | T | \cdot \exp\left( \frac{- (r^2 - s^2)^2}{t^2 \cdot ( r^2 + s^2) } \right) +  \P_\mu ( \hat c_{I\cap T} >s )  + \P_\mu ( \bar a_{I\cap T}> t ). \label{F_split_bound}
%&\qquad \qquad + \P_\mu ( \| D_{I\cap T} H \wm D_{T^c} \|>s )  + \P_\mu ( \|D_{I\cap T}\bar H D_{T^c}\|_{2,1}> t )
\end{align}
To estimate the last two terms, we rearrange $p$, $H$, $\bar H$ and $M$ again, to now have $T= \{1,\ldots ,N\}$, for $N=|T|$.
To estimate
\begin{align*}
\P_\mu ( \hat c_{I\cap T} >s ) = \P_\mu(\| D_{I\cap T}\hat H \wm D_{T^c} \| > s )=\P_\mu(\| D_I D_T \hat H \wm R_{T^c} \| > s ),
\end{align*}
we now abbreviate $B = (D_T \hat H \wm R_{T^c})^\tp = (B_1,\ldots ,B_N ,0\ldots,0)\in \R^{(K-N) \times K}$, and again rewrite the expression we want to bound as a martingale. Using the random variables $\chi_n$, the filtration $\mathcal{F}_n = \sigma(\chi_1, \ldots, \chi_n)$, as well as $\pi_n=\E_\mu[\chi_n\mid\mathcal{F}_{n-1}]$, where $\pi_n \leq p_n$ (a.s.), and $X_n=(\chi_n -\pi_n) B_nB_n^\tp$, then yields
\begin{align*}
\P_\mu ( c_{I\cap T} >s )&=\P_\mu (\| D_I B^t \| > s ) =\P_{\mu} (\|B D_I B^t\|> s^2) \\
&= \P_\mu\left( \|\sum_{n = 1}^N \chi_n B_n B_n^\tp \| \geq s^2 \right )\\
&\leq \P_\mu \left( \|\sum_{n = 1}^N (\chi_n -\pi_n) B_n B_n^\tp \| \geq s^2 - \|\sum_{n = 1}^N \pi_n B_n B_n^\tp\| \right ) \\
&\leq \P_\mu \left( \|\sum_{n = 1}^N X_n\| \geq s^2 - \| B \wm \|^2 \right )\leq \P_\mu \left( \|\sum_{n = 1}^N X_n\| \geq s^2 - \| \wm \hat H \wm \|^2 \right ).
\end{align*}
Similar to before we next bound the increments $X_n$ and the quadratic variation process as
\begin{align*}
&\| X_n\| \leq \|B_n\|_2^2 \leq \max_{n \in T} \|R_{T^c}^\tp \wm \hat H D_T e_n\|_2^2 \leq \max_{n} \|\wm\hat H e_n\|_2^2 = \|\wm\hat H\|_{2,1}^2\\
\mbox{and} \quad &\| \sum_{n=1}^N \E[ X_n^2 \mid \mathcal{F}_{n-1}] \| \leq \max_{n \in T}  \|B_n\|_2^2  \cdot \| \sum_{n=1}^N \pi_n B_n B_n^\tp \| 
\leq \|\wm\hat H\|_{2,1}^2 \cdot \|\wm \hat H \wm\|^2 \quad \mbox{(a.s.)},
\end{align*}
so abbreviating $\hat a = \|\wm\hat H\|_{2,1}$ and $\hat c = \|\wm \hat H \wm\|$, Freedman's inequality yields,
\begin{align}
\P_\mu ( c_{I\cap T} >s )=\P_\mu ( \| D_{I} D_T H \wm R_{T^c} \|>s ) &\leq  | T^c | \cdot \exp\left( \frac{- (s^2 - \hat c^2)^2}{ \hat a^2 \cdot ( s^2 + \hat c^2) }\right). \label{c_bound}
\end{align}
% \begin{align}
% \P_\mu ( \bar c_{I\cap T} >s )=\P_\mu ( \| D_{I} D_T \bar H \wm D_{T^c} \|>s ) &\leq  | T^c | \cdot \exp\left( \frac{-(s^2 - \bar c^2)^2}{\bar a^2 \cdot ( s^2 + \bar c^2) }\right).
% \end{align}
Finally, to estimate $\P (\bar a_{I\cap T} >t ) = \P(\|D_{I}D_T \bar H D_{T^c}\|_{2,1} >t)$ we first bound
\begin{align*}
    \P_\mu(\|D_{I}D_T \bar H D_{T^c}\|_{2,1} >t) &= \P_\mu(\max_{k \in T^c} \|D_{I}D_T \bar H e_k\|_2 >t) \leq \sum_{k \in T^c} \P_\mu (\|D_{I}D_T \bar H e_k\|_2 >t).
\end{align*}    
We then set $X_n = (\chi_n-\pi_n)\bar H_{nk}^2$, meaning $|X_n|\leq \|\bar H \|_{\infty, 1}=: \mu_0^2$, and get
\begin{align*}
 \sum_{n=1}^N \E[ X_n^2 \mid \mathcal{F}_{n-1}] \leq  \sum_{n=1}^N p_n \bar H_{nk}^4 \leq \mu_0^2\cdot  \sum_{n=1}^N p_n \bar H_{nk}^2 = \mu_0^2 \cdot \|\wm \bar H e_k \|_2^2 \leq  \mu_0^2 \cdot\| \wm \bar H\|^2_{2,1}\quad \mbox{(a.s.)}.
\end{align*}
Using the usual procedure and the bound $\| \wm \bar H\|_{2,1} \leq \| \wm \hat H\|_{2,1} =\hat a$, then yields
\begin{align*}
 \P_\mu (\|D_{I}D_T \bar H e_k\|_2 >t) &= \P_\mu \left(\sum_{n=1}^N \chi_n \bar H_{nk}^2 >t^2\right) \\
  &= \P_\mu \left(\sum_{n=1}^N  X_n + \sum_{n=1}^N \pi_n \bar H_{ni}^2 >t^2\right)\\
  &\leq \P_\mu \left(\sum_{n=1}^N  X_n >t^2 - \hat a^2 \right)\leq \exp\left( \frac{-(t^2 - \hat a^2)^2}{ \mu_0^2 \cdot ( t^2 +  \hat a^2) }\right),
\end{align*}
or in other words 
\begin{align}
   \P (\bar a_{I\cap T} >t ) = \P_\mu(\|D_I D_T \bar H D_{T^c}\|_{2,1} >t) \leq |T^c| \cdot  \exp\left( \frac{-(t^2 - \hat a ^2)^2}{ \mu_0^2 \cdot ( t^2 +  \hat a^2) }\right).\label{a_bound}
\end{align}
Plugging the bound above as well as the bound for $\P_\mu ( \hat c_{I\cap T} >s )$ in \eqref{c_bound} into \eqref{F_split_bound} then yields for $r,s,t$ with $\hat c < s < r$ and $t > \hat a$
\begin{align}
&\P_{\mu} \left( \| D_{I}\,D_T H D_{T^c}\,D_{I} \| \geq r \right ) \notag \\
&\qquad \leq  | T | \cdot \exp\left( \frac{-(r^2 - s^2)^2}{ t^2 \cdot ( r^2 + s^2) }\right) +  | T^c | \cdot \exp\left( \frac{-(s^2 - \hat c^2)^2}{ \hat a^2 \cdot ( s^2 +  \hat c^2) }\right)+|T^c| \cdot  \exp\left( \frac{-(t^2 - \hat a ^2)^2}{ \mu_0^2 \cdot ( t^2 + \hat a^2) }\right),\notag
%&\qquad \leq  \frac{3K}{2} \max\left\{\exp\left( \frac{- (r^2 - s^2)^2}{t^2 \cdot ( r^2 + s^2) }\right), \exp\left( \frac{- (s^2 - \bar c^2)^2}{ \bar a^2 \cdot ( s^2 +  \bar c^2) }\right), \exp\left( \frac{- (t^2 - \bar a^2)^2}{ \mu_0^2 \cdot ( t^2 + \bar a^2) }\right)\right\}, 
%&\qquad \qquad + \P_\mu ( \| D_{I\cap T} H \wm D_{T^c} \|>s )  + \P_\mu ( \|D_{I\cap T}\bar H D_{T^c}\|_{2,1}> t )
\end{align}
and it only remains to choose $r,s,t$ wisely. To do that observe that whenever $r > \hat c$ we can find $n>1$ and $\rho>0$ such that
\begin{align*}
    t^2 &:=   \max\{ \rho \mu_0^2, n \hat a^2\} \notag \\
    s^2 &:=   \max\{ \rho \hat a^2, n\hat c^2\} \notag \\
    r^2 &:=   \max\{ \rho  t^2, n s^2\}   = \max\{\rho^2  \mu_0^2, \rho n \hat a^2, n^2 \hat c^2\}. %\label{r_cond} %
\end{align*}
We then have 
\begin{align*}
\frac{(r^2 - s^2)^2}{ t^2 \cdot ( r^2 +  s^2) }& = \frac{r^2\cdot \frac{n-1}{n} + r^2\cdot \frac{1}{n} - s^2}{ t^2}  \cdot  \frac{r^2 - s^2}{r^2 +  s^2} \\
&\geq \frac{r^2 (n-1)}{ n t^2}  \cdot  \frac{n-1}{n+1} \geq \rho \cdot \frac{ (n-1)^2}{n(n+1)} = : \rho \cdot \kappa(n),
\end{align*}
meaning all three exponentials are bounded by $e^{-\rho \kappa(n)}$. 
In particular we can set
$$\rho =  \min\left\{ \frac{r}{\mu_0}, \frac{r^2}{n \hat a^2}\right\} ,$$
meaning that for all $r>n\hat c$ we have
\begin{align}
&\P_{\mu} \left( \| D_I \,D_T H D_{T^c} \,D_I \| \geq r \right ) \leq (K + |T^c|) \exp\left( - \min\left\{ \frac{r}{n \mu_0}, \frac{r^2}{n^2 \hat a^2}\right\} \cdot \frac{(n-1)^2}{n+1}\right). \label{quasi_final_bound_Tc}
\end{align}
Since $|T^c| \leq K/2$  we finally get for all $r \geq 2 n \hat c $
\begin{align*}
    \P_{\mu} \left( \| D_I H D_I \| \geq r \right ) &
    \leq 54 K \exp\left( - \min\left\{ \frac{r}{2n \mu_0}, \frac{r^2}{4n^2 \hat a^2}\right\} \cdot \frac{(n-1)^2}{n+1}\right).
\end{align*}
To get to the bound for the non-symmetric case, note that
$\|A\| = \|A^t\|$, so
\begin{align*}
    &\P_\mu\left( \max \{ \| D_I \, D_T \hol D_{T^c} , D_I \|, \| D_I \, D_{T^c} \hol D_{T} \, D_I \| \}\geq r\right)\\
    &\hspace{4cm}\leq \P_\mu\left( \| D_I \, D_T \hol D_{T^c} \, D_I \|\geq r\right)+\P_\mu\left( \| D_I \, D_T \hol^\tp D_{T^c} \, D_I \|\geq r\right).
\end{align*}
Finally applying \eqref{quasi_final_bound_Tc} to $H^\tp$, uniformly bounding the resulting exponents and setting $n=3$ yields the desired result.
\end{proof}

%\newpage
\appendix

%%%%%%%%%%%%%%
\noindent
\section{Proof of Lemma~\ref{cond_on_F} and \ref{bound_inclusion_prob}}\label{sec:proof_lemma}

For the proofs it will be a little bit more intuitive to work with the sets $I(\delta)$ rather than with $\delta$ directly. We therefore extend the multi-index notation to sets and define
$p^J:=\prod_{i\in J} p_i$ for $J\subseteq [K]$. This means that if we have a disjoint union of two sets $J,F$ meaning $J\cap F=\emptyset$, which we denote by $J\dcup F$, we get $p^{J\dcup F} = p^J \cdot p^F$. Denoting the set complement as $J^c =[K] \setminus J$, leads to
\begin{align*}
    \P(I = J) = p^J(\one -p)^{J^c} \quad
   \mbox{and} \quad  \P(A) = \sum_{\delta \in A} p^{I(\delta)} (\one -p)^{I(\delta)^c} \quad \mbox{for all } A\subseteq \Delta.
\end{align*}  
Further, recall that $\mu(\delta) = \mu(I(\delta))$ and that we call a set incoherent if $\mu(J)=\one_J^t M \one_J=0$, and a set $J$ or index $k$ incoherent to a set $F$ if $\mu(F,J)=\one_F^t M \one_J=0$ and $\mu(F,k)=\mu(F,\{k\})= 0$ respectively. Whenever $\mu(\cdot,\cdot) > 0$ we refer to the set or pair as coherent.\\

\begin{proof}[of Lemma \ref{cond_on_F}] We want to show that for
 a real-valued random variable $g$ on $\mathcal{P}([K])$, a set $F\subseteq S\subseteq [K]$ with $\mu(F) = 0$ and $A_F = \{ k \in S^c \mid \mu(F,k)=0\}$, we have for all $r\in \R$
\begin{align*}
    \P_\mu(g(I) \geq r \mid I\cap S = F) = \P_\mu^{A_F}(g(I\cup F)\geq r).
\end{align*}
Let $C_F = (A_F \cup S)^c= \{k\in S^c: \mu(F,k) > 0\}$ be the set of indices in $S^c$ that are coherent with the selected set $F$ and thus cannot be added, meaning $\P_\mu(I=J\cup F)=0$ for $J\subset S^c$ unless $J\cap C_F =\emptyset$. 
We first show that 
\begin{align*}
\P(\underbrace{g(I)\geq r, I\cap S = F, \mu=0}_{=:B\subseteq \Delta}) = \underbrace{p^F  (\one-p)^{S\setminus F} (\one -p)^{C_F}}_{=:\kappa(F)} \cdot\; \P^{A_F}(g(I\cup F) \geq r, \mu = 0).
\end{align*}
Now note that we can only have $I(\delta) \cap S = F$ if there exists some set $J$ with $J\subseteq S^c$ such that $I(\delta)= F \dcup J$ or equivalently there exists some $\eta$ with $I(\eta) \subseteq S^c$ such that $\delta =\one_F + \eta$.
\begin{align*}
   B&= \{\delta: g(I(\delta)) \geq r, I(\delta) \cap S = F, \mu(I(\delta)) = 0 \}\\
   &= \{ \one_F + \eta: g(F\cup I(\eta))\geq r, I(\eta)\subseteq  S^c, \mu(\one_F+ \eta) =0\}.
\end{align*}
Next note that for $F$ with $\mu(F) =\mu(\one_F)=0$ we have
\begin{align*}
   0 = \mu(\one_F + \eta) = \mu(\one_F) + 2\mu(\one_F,\eta) + \mu(\eta) = 2\mu(F,I(\eta)) + \mu(\eta) 
\end{align*}
if and only if $I(\eta)\subseteq A_F$ and $\mu(\eta) = 0$, so we get 
\begin{align*}
   B&= \{ \one_F + \eta: g(F\cup I(\eta))\geq r, I(\eta)\subseteq  S^c, \mu(\one_F+ \eta) =0\}\\
   &= \{ \one_F + \eta: g(F\cup I(\eta))\geq r, I(\eta)\subseteq  A_F, \mu(\eta) =0\}.
\end{align*}
Finally, note that if $J \subseteq  A_F$ then $(F \dcup J)^c = (S\dcup A_F \dcup C_F)\setminus (F\dcup J) = (S\setminus F) \dcup (A_F\setminus J) \dcup C_F$, so we get
\begin{align*}
\P(B) &= \sum_{\substack{\delta: g(F\cup I(\delta))\geq r\\ I(\delta)\subseteq  A_F,\\ \mu(\delta) =0 }} p^{F \cup I(\delta)} (\one-p)^{[F \cup I(\delta)]^c}\\
&= \sum_{\substack{\delta: g(F\cup I(\delta))\geq r\\ I(\delta)\subseteq  A_F,\\ \mu(\delta) =0 }} p^{F} p^{I(\delta)} (\one-p)^{S\setminus F} (\one-p)^{A_F\setminus I(\delta)}(\one-p)^{C_F}\\
&= p^{F}(\one-p)^{S\setminus F} (\one-p)^{C_F} \cdot  \sum_{\substack{\delta: g(F\cup I(\delta))\geq r\\ I(\delta)\subseteq  A_F,\\ \mu(\delta) =0 }}  p^{I(\delta)}  (\one-p)^{A_F\setminus I(\delta)}\\
& =\kappa(F) \cdot \P^{A_F}( g(F\cup I) \geq r, \mu= 0).
\end{align*}
Setting $g = 1$ and $r=0$ immediately yields $\P(I\cap S = F, \mu=0) = \kappa(F) \cdot \P^{A_F}(\mu = 0)>0$, since $\mu(0) = 0$ and $I(0) = \emptyset \subseteq A_F $, so we get 
\begin{align*}
    \P_\mu(g(I) \geq r \mid I\cap S = F) %&= \frac{\P_\mu (g(I) \geq r, I\cap S = F )}{\P_\mu (I\cap S = F )} \\
    &=\frac{\P (g(I) \geq r, I\cap S = F, \mu = 0)}{\P(I\cap S = F, \mu=0)}\\
     &=\frac{\kappa(F) \cdot \P^{A_F}( g(F\cup I) \geq r, \mu= 0)}{\kappa(F) \cdot \P^{A_F}(\mu = 0)}= \P_\mu^{A_F}( g(F\cup I) \geq r).
\end{align*}
\end{proof}
%
% \begin{lemma}
% Let $\chi_\ell: \Delta\to \R, \delta \to \delta_\ell$ then
% for any $S\subseteq [K]$ we have
% \begin{align*}
%     \E^S_\mu(\chi_\ell) = \P^S_\mu(\chi_\ell = 1) = \P^S_\mu(\ell \in I) \leq p_\ell.
% \end{align*}
% \end{lemma}
\begin{proof}[of Lemma~\ref{bound_inclusion_prob}]
We want to bound the inclusion probabilities of $\P_\mu^S$ by those of $\P$, meaning for $\chi_\ell: \Delta\to \R, \delta \mapsto \delta_\ell$ and any $S\subseteq [K]$ we have
\begin{align*}
    \E^S_\mu(\chi_\ell) = \P^S_\mu(\chi_\ell = 1) = \P^S_\mu(\ell \in I) \leq p_\ell.
\end{align*}
For $\ell \notin S$ we have $\P_\mu^S (\ell \in I)=0 \leq p_\ell$. Let $A= \{j \in S : M(\ell,j) = 0 \}$ and $C = S\setminus (A\cup \{\ell\})$. Now if $\delta$ with $I(\delta)\subseteq S$, $\delta_\ell =1$ and $\mu(\delta) =0$ we need that $I(\delta) \cap C =\emptyset$, or in other words we can write $\delta$ as $\delta = \eta + \one_{\ell}$ where $I(\eta) \subseteq A$, meaning
\begin{align*}
B^1_\ell &= \{\delta: \mu(\delta) =0, I(\delta) \subseteq S, \delta_\ell = 1\}\\
&= \{\delta: \mu(\delta) =0, I(\delta) \subseteq A\cup\{\ell\},\delta_\ell = 1\}= \{\eta + \one_{\ell}: \mu(\eta) =0, I(\eta) \subseteq A\}.
\end{align*}
On the other hand we have more incoherent subsets of $S$ than incoherent subsets of $A\cup \{\ell\}$, which we can in turn divide into those that contain $\ell$ and those that do not, meaning
\begin{align*}
B &= \{\delta: \mu(\delta) =0, I(\delta) \subseteq S\}\\
&\supseteq \{\delta: \mu(\delta) =0, I(\delta) \subseteq A\cup\{\ell\}\\
&= \{\delta: \mu(\delta) =0, I(\delta) \subseteq S, \delta_\ell = 1\} \dcup \{\delta: \mu(\delta) =0, I(\delta) \subseteq S, \delta_\ell = 0\} \\&=B^1_\ell \dcup \{\eta : \mu(\eta) =0, I(\eta) \subseteq A\}=: B^1_\ell \dcup B^0_\ell.
\end{align*}
Next note that 
\begin{align*}
    \P(B^1_\ell) &= \sum_{\eta:\mu(\eta)=0, I(\eta)\subseteq A } \: p^{I(\eta)}\cdot p_\ell \cdot (\one-p)^{A\setminus I(\eta)} \cdot (\one-p)^{S^c \cup C} \\
    &= \frac{p_\ell}{1-p_\ell} \cdot \sum_{\eta:\mu(\eta)=0, I(\eta)\subseteq A } \: p^{I(\eta)} \cdot (\one-p)^{A\setminus I(\eta)} \cdot (1-p_\ell) \cdot (\one-p)^{S^c \cup C} = \frac{p_\ell}{1-p_\ell} \cdot \P(B^0_\ell),
\end{align*}
leading to 
\begin{align*}
    \P^S_\mu(\ell \in I) = \frac{\P(B^1_\ell)}{\P(B)}\leq \frac{\P(B^1_\ell)}{\P(B^1_\ell) + \P(B^0_\ell)} = p_\ell
\end{align*}
\end{proof}

%\acks{This work was supported by the Austrian Science Fund (FWF) under Grant no.~Y760.}
% We would like to thank the anonymous reviewers for constructive criticism (hopefully)
% or scientific misconduct by making us cite all their irrelevant papers (likely). 
% We are immensely grateful to the designated
% proof-reading bunny for finding many mistakes.

%%% put bibliography path here 

\end{document}